\documentclass{article}

\usepackage{amsmath, bbm, amsfonts,  comment, amsthm, amssymb}
\usepackage{hyperref}
\hypersetup{
    colorlinks=true,
    linkcolor=blue,
    citecolor=red,
    urlcolor=blue,
    pdfborder={0 0 0}
}    
\usepackage{graphicx}
\usepackage{xcolor}

\usepackage{cleveref}

  \crefname{theorem}{Theorem}{Theorems}
  \crefname{thm}{Theorem}{Theorems}
  \crefname{lemma}{Lemma}{Lemmas}
  \crefname{lem}{Lemma}{Lemmas}
  \crefname{remark}{Remark}{Remarks}
  \crefname{prop}{Proposition}{Propositions}
  \crefname{proposition}{Proposition}{Propositions}
\crefname{notation}{Notation}{Notations}
\crefname{claim}{Claim}{Claims}
  \crefname{defn}{Definition}{Definitions}
  \crefname{corollary}{Corollary}{Corollaries}
  \crefname{section}{Section}{Sections}
  \crefname{figure}{Figure}{Figures}
  \crefname{exercise}{Exercise}{Exercises}
    \crefname{assumption}{Assumption}{Assumptions}

\newtheorem{thm}{Theorem}[section]

\newtheorem{lemma}[thm]{Lemma}
\newtheorem{corollary}[thm]{Corollary}

\newtheorem{defn}[thm]{Definition}

\numberwithin{equation}{section}

\theoremstyle{definition}

\def\cS{\mathcal{S}}

\def\P{\mathbb{P}}

\def\Z{\mathbb{Z}}

\def\1{\mathbf{1}}

\DeclareMathOperator{\w} {\mathsf{w}}

\newcommand{\note}[1]{{\color{red}{[note: #1]}}}

\usepackage{extarrows}

\title{Forests on wired regular trees}
\author{Gourab Ray\thanks{University of Victoria. Research supported by NSERC 50311-57400} \and Ben Xiao \thanks{University of Victoria}}
\date{July 2021}

\newcommand{\p}{\mathbbm{P}}

\newcommand{\ext}{Z^{\times}}
\begin{document}
\maketitle

\begin{abstract}
The Arboreal gas model on a finite graph $G$ is the Bernoulli bond percolation on $G$ conditioned on the event that the sampled subgraph is a forest.  In this short note  we study the arboreal gas on a regular tree wired at the leaves and obtain a comprehensive description of the weak limit of this model.
\end{abstract}

\section{Introduction}\label{sec:intro}
Let $G = (V,E)$ be a finite graph. A forest is a possibly disconnected graph with no cycles. A connected forest is a tree. The \textbf{Arboreal gas} model with parameter $p \in [0,1)$ is a probability measure on $\{0,1\}^{E}$ defined as 
\begin{equation}
\P_{G, p}(f) = \frac{p^{o(f)}(1-p)^{c(f)}\mathbbm{1}_{\text{f is a forest}}}{Z_{G,p}} , \qquad f \in \{0,1\}^{E}\label{eq:gas} ,
\end{equation}
where $o(f)$ is $\{e\in E: f(e)=1\}$ (called \textbf{open} edges) and $c(f) = E \setminus o(f)$ (called \textbf{closed edges}) and  $Z_{G,p}$ is the appropriate partition function. Note that if $G$ has no cycles then this model is simply the Bernoulli bond percolation model with parameter $p$. As $p \to 1$ this model converges to the uniform spanning tree on $G$. Although the uniform spanning tree on finite and infinite graphs has been studied in great depth, (we refer to \cite{lyons2017probability} and references therein for the interested reader) the Arboreal gas model is much less studied even when $p$ is close to 1. The reason for this stems from the fact that the tools used to study uniform spanning tree, viz Wilson's algorithm and Electrical network theory is no longer applicable for any $p<1$.

The goal of this note is to study the Arboreal gas model on a regular tree which we define now. 
Fix an integer $d \ge 2$ and let $T_{n,d}$ denote the ordered $d$-ary tree of \textbf{depth} $n$ described as follows: $T_{n,d}$ is a tree where all degrees are $d$ except one vertex which has degree $d-1$ (called the \textbf{root vertex}, denoted $\rho$) and  the leaves (degree 1 vertices) and furthermore all leaves are at graph distance exactly $n$ from $\rho$. Let $T^{\w}_{n,d}$ denote the tree $T_{n,d}$ where all the leaves are identified into a single vertex denoted by $\partial$ (this enforces a \textbf{wired boundary}). We denote the vertex set of $T^{\w}_{n,d}$ by $V_{n,d}$, the edge set by $E_{n,d}$.

 Let $T_d$ denote the infinite $d$-ary tree, i.e., it is the unique infinite tree with where every vertex has degree $d$ except one vertex, also called the root (denoted $\rho$). Note that if we identify all vertices which are at distance strictly more than $n-1$ from the root into a single vertex and remove all the self loops then we obtain the graph $T_{n,d}^{\w}$. We now state the main result of this note. We write $\P_{n,d,p}$ in short for $\P_{T_{n,d}^{\w},p}$.

\begin{thm}\label{thm:main}
Fix $p \in [0,1)$. As $n \to \infty$, $\P_{n,d,p}$ converges weakly to a limiting law $\P_{d,p}$. Furthermore the law $\P_d$ can described explicitly as a hidden tree indexed Markov chain model.
\end{thm}
 
The explicit description of the tree indexed Markov chain model appears later in \cref{thm:main_improved}. For now let us mention that by storing the extra information of whether an open edge is connected to the leaf or not turns the model into a process with nice Markovian properties. This idea is inspired by a paper of H\"aggstr\"om \cite{haggstrom1996random}. This also allows an explicit description of the limiting law which is the content of the next theorem.
\medskip
\begin{thm}\label{thm:description}
The limiting law $\P_{d,p}$ undergoes a phase transition as follows.

\begin{itemize}
\item  For $p \le 1/d$, $\P_{d,p}$ is an i.i.d.\ Bernoulli percolation on $T_d$ with parameter $p$. 
\item If $p \in (1/d,1)$ then $\P_{d,p}$ has infinitely many infinite components. Furthermore each infinite component is one ended and can be described as follows. 
\begin{itemize}
\item Suppose $e = (u,v)$ be an edge with $u$ being closer to the root. Conditioned on $e$ being closed, $v$ has a probability $q = \frac{dp - 1}{p(d - 1)}$ of being in an infinite component which can be sampled as follows. Declare a child edge $(v,v_1)$ of $v$ to be open uniformly at random, then inductively having defined $v_i$, declare a child edge $(v_i,v_{i+1})$ to be open uniformly at random. This defines an infinite path $P$.
\item Each vertex in such a path $P$ has an independent $Bin(d-1,d)$ many edges attached to it which do not belong to $P$ obtained by doing Bernoulli percolation on the $d-1$ edges with parameter $1/d$. The other endpoints of these edges are attached to almost surely finite trees which are distributed as independent critical branching processes with offspring distribution given by Bin$(d,1/d)$, and obtained by sampling a Bernoulli bond percolation with parameter $1/d$. 
\end{itemize}
 \item Furthermore, all the finite components are also distributed as independent critical branching process with offspring distribution Bin$(d,1/d)$.
\end{itemize}
\end{thm}

The random forest model on complete graphs was studied in \cite{luczak1992components,martin2017critical} and more recently in the regular lattice $\Z^d$ in \cite{bauerschmidt2021random} (for $d=2$) and \cite{bauerschmidt2021percolation} (for $d \ge 3$). We refer to these papers for more references and history of the model. Let us mention that both the papers \cite{bauerschmidt2021random,bauerschmidt2021percolation} use a connection between the two point function in the Arboreal gas model and a non linear sigma model with hyperbolic target space. In contrast, our techniques are quite elementary and relies heavily on the regularity of the tree which leads to an explicit recursion. What is perhaps surprising is that this recursion can be explicitly solved.

The results we have are in line with the curious phenomenon also present in \cite{bauerschmidt2021random,bauerschmidt2021percolation}: in the supercritical phase, the finite components behave like a critical component, in particular their diameter and volume have polynomial tails, which is in sharp contrast with more standard models like Bernoulli percolation. However there is one observation which might only be seen in the tree or in general in a nonamenable graph. The Bernoulli percolation model with parameter $p$ can be seen as the `free' Arboreal gas model on the tree. It is well known that in the supercritical phase of a branching process, the finite components  behave like a subcritical branching process and their volume has exponential tail \cite{athreya2004branching}. This is in contrast to the behaviour in the wired tree where all finite clusters behave like a critical branching process. Comparing the behaviour of the wired and the free Arboreal gas models in a general graph remains a natural open question.

While preparing this manuscript it has come to our attention that P. Easo has been independently and simultaneously working on a similar problem \cite{easo}, although his proof takes a different approach than ours.

\section{Recursion}\label{sec:recursion}
Given a forest $f  = (f_e)_{e \in E_{n,d}} \in \{0,1\}^{E_{n,d}}$,  We say $x$ is connected to $y$ in $f$ there is an open path in $f$ connecting $x$ and $y$ and denote this by $x \leftrightarrow y$. Also given an edge $e$, let $e_-$ denote the vertex of $e$ closer to the root $\rho$ and let $e_+$ be the vertex further to the root. The edges attached to $e_+$ are called the \textbf{children edges} of $e$ and the edge attached to $e^-$ is called the \textbf{parent edge} of $e$. The children edges of the parent edges of $e$ except $e$ are called the \textbf{siblings} of $e$. Note that every edge has $d-1$ siblings.

We start with a few definitions. 
\begin{itemize}
\item If $f$ is a forest in $T^{\w}_{n,d}$, then define $w(f) = p^{o(f)}(1-p)^{c(f)}$ to be the weight of $f$.
\item We define $Z_{n,d,p}$ to be the partition function associated with the Arboreal gas measure $\P_{n,d,p}$. We write $Z_{n,d,p} = Z_{n,d,p}^{S} + \ext_{n,d,p}$ where $Z_{n,d,p}^{S}$ is the sum over the weights of all forests where $\rho \leftrightarrow \partial$ and $\ext_{n,d,p}$ is the sum over the rest. 

\item We define $q_{n,d,p} = \P_{n,d,p}(0 \leftrightarrow \partial )$. Note that in this notation $$q_{n,d,p} = \frac{Z_{n,d,p}^S}{Z_{n,d,p}}; \qquad 1-q_{n,d,p} = \frac{\ext_{n,d,p}}{Z_{n,d,p}}.$$ We think of this quantity as the probability of `survival' if we interpret the Arboreal gas as a branching process.
\end{itemize}

\begin{lemma}\label{lem:recursion_partition}
Fix $d, p \in [0,1)$. We have
$$Z_{n,d,p}^S = dpZ_{n-1,d,p}^S((1-p)Z_{n-1,d,p}^S + Z_{n-1,d,p}^{\times})^{d-1} \text{ and }\ext_{n,d,p} = ((1-p)Z_{n-1,d,p}^S + \ext_{n-1,d,p})^d.$$
\end{lemma}
\begin{proof}
We drop $d,p$ from the subscript in the notation for simplicity.
The only way to survive is if $\rho$ is connected to $\partial$ through a path which goes through one of the edges attached to the root and none of the other edges attached to the root reach $\partial$. The child which survives contributes a weight $pZ^S_{n-1}$. To take care of the other children, we need to select $i$ many to be open but must lead to extinction and $d-1-i$ many to be closed but the paths may or may not survive. This has weight $$\sum_{i = 0}^{d-1}\binom{d-1}{i} (p\ext_{n-1})^i((1-p)Z_{n-1})^{d-1-i}.$$ Overall, we see that $$Z_{n}^S = dpZ^S_{n-1}\sum_{i = 0}^{d-1}\binom{d-1}{i} (p\ext_{n-1})^i((1-p)Z_{n-1})^{d-1-i}.$$ Applying the binomial theorem yields $$Z_{n}^S = dpZ_{n-1}^S((1-p)Z_{n-1} + pZ_{n-1}^{\times})^{d-1}.$$ Substituting $Z_{n-1} = Z_{n-1}^S + \ext_{n-1} $ and simplifying we get the desired expression as claimed.

For $\ext_n$, we need to pick $i$ many open edges but they must lead to extinction and $d-i$ many closed edges and they may or may not be extinct and so we get $$\sum_{i = 0}^{d}\binom{d}{i} (p\ext_{n-1})^i((1-p)Z_{n-1})^{d-i}$$ and so we arrive again at our desired expression after applying the Binomial theorem and simplifying.
\end{proof}
Our goal is to now solve the family of recursions $Z_n^S$ and $\ext_n$ with the initial conditions $Z_0^S = 1$ and $\ext_0 = 0$ by convention (for $n=0$, the root is connected to itself). Set
$$
K_{n,d,p}  =\frac{\ext_{n,d,p}}{Z^S_{n,d,p}} = \frac{(1-p)Z_{n-1,d,p}^S + Z_{n-1,d,p}^{\times} }{ dpZ_{n-1,d,p}^S}.
$$
where the last equality is a consequence of \cref{lem:recursion_partition}. Note that $Z_{n,d,p}^S$ is $0$ only when $p = 0$ and since that case is not so interesting, we will from now on assume $p > 0$.

\begin{lemma}\label{lem:recursion_2}
Set $p \in (0,1), n \ge 0.$ Then
 $$K_{n,d,p} = \frac{(1-p)}{(dp)^n} \sum_{i = 0}^{n-1}(dp)^i $$
 Consequently if $p \neq 1/d$ then
 $$
 K_{n,d,p}  =  \frac{1-p }{dp - 1} (1-(dp)^{-n})
 $$
and if $p = \frac1d$,
$$
K_{n,d,p} = \left(1-\frac1d\right)n.
$$
\end{lemma}
\begin{proof}
We drop the subscripts $d,p$ from the notation. Using \cref{lem:recursion_partition}, we can write $$K_n = \frac{(1-p)Z_{n-1}^S + Z_{n-1}^{\times}}{dpZ_{n-1}^S} = \frac{1-p}{dp} + \frac{\ext_{n-1}}{dpZ_{n-1}^S} = \frac{1-p}{dp} + \frac{1}{dp}K_{n-1}.$$  We claim now the first equality holds and we prove it by induction. Since $K_0 = 0$ by convention and the right hand side is $0$ as well (empty sum is 0), the base case holds. Assume the formula is true for $K_n$. Then $$K_{n+1} = \frac{1-p}{dp} + \frac{1}{dp}K_{n} =  \frac{1-p}{dp} + \frac{1}{dp}\frac{(1-p)}{(dp)^n} \sum_{i = 0}^{n-1}(dp)^i  = \frac{1}{(dp)^{n+1}}[(1-p)(dp)^n + (1-p) \sum_{i = 0}^{n-1} (dp)^i]$$ and factoring out the $1-p$ terms gives us the desired expression. By induction, the formula holds for all $n$. The second equality comes from applying the geometric series formula to the sum.
\end{proof}
An easy consequence of \cref{lem:recursion_2} is that
\begin{corollary}\label{cor:survival_prob}
If $p \le 1/d$, $q_{n,d,p} \to 0$. On the other hand if $p \in (1/d,1)$,  \begin{equation}
q_{n,d,p} \to \frac{dp - 1}{p(d - 1)} =:q_{d,p}\label{eq:survival}
\end{equation}
\end{corollary}
We call $q_{d,p}$ the \textbf{survival probability} drawing an analogue from the classical branching process.
It also follows from the expression in \cref{lem:recursion_2} that if $p <1/d$ then $q_{n,d,p}$ converges to 0 exponentially fast, while if $p = 1/d$, $q_{n,d,p}= d/(d-1)n^{-1}$. This is reminiscent of similar results for branching processes and indeed resonates with \cref{thm:description} that if $p \le 1/d$ then the Arboreal gas converges to an i.i.d.\ Bernoulli$(p)$ process.

\section{A hidden Markov model}
This section is inspired by the ideas in H\"aggstr\"om \cite{haggstrom1996random}.
Given a forest $f  = (f_e)_{e \in E_{n,d}} \in \{0,1\}^{E_{n,d}}$ in we introduce a bijective mapping $\phi : \{0,1\}^{E_{n,d}} \to \{0',1',2'\}^{E_{n,d}}$ as follows. 
\begin{itemize}
\item If $f_e = 0$, $\phi$ always maps it to $0'$.
\item If $f_e =1$, $\phi $ maps it to $2'$ if $e_+ \leftrightarrow \partial$ in $f$ and to $1'$ otherwise.
\end{itemize}
Note that while $\phi$ depends on global properties of $f$, $\phi^{-1}$ is a local map which maps $\{1',2'\}$ to $1$ and $0' $ to 0. Denote the pushforward of $\P_{n,d,p}$ under $\phi$ by $\P'_{n,d,p}$.
We now show that this mapping converts the Arboreal gas model into a Markovian model with states $\{0',1',2'\}$. We record  this observation in the following lemma. Fix an edge $e \in E_{n,d}$ and note that removing $e$ breaks up $T_{n,d}$ into two components, call the component containing the root $C_\rho  = (V(C_\rho), E(C_\rho))$. Suppose $\xi \in \{0',1',2'\}^{E(C_\rho)}$ which is valid in the sense that $\P'_{n,d}$ probability of $C_\rho$ having configuration $\xi$ is strictly positive. Number the children edges $(e_1,e_2, \ldots, e_d)$ ordered from left to right. For any element $f' \in  \{0',1',2'\} ^{E_{n,d}}$ and $S \subset E_{n,d}$, let $f'|_{S}$ denote the restriction of $f'$ to $S$. Let $\cS =(e_1,e_2, \ldots, e_d) $. Suppose $e_+$ is at graph distance $k<n$ from $\rho$. In the following lemma, we drop $d$ from the subscripts of the partition functions. We again drop $d,p$ from the subscripts in the following lemma for brevity.

\begin{lemma}\label{lem:Markovian}
Fix $p \in [0,1)$ and $e, \cS, \xi, C_\rho$ be as above and let $m = n-k$. Suppose $\P'_{n,d,p} (f'|_{C_\rho}=\xi) >0$. Then the following holds for all $e_j$ and for all $\cS_1 \subseteq \cS \setminus \{e_j\}$.
\begin{multline}
\P'_{n,d,p} \left(f'_{e_j} =2' , f'_{g} =1' \forall g \in \cS_1 , f'_{h} = 0' \forall  h \in \cS \setminus (\cS_1 \cup \{e_j\}) \Big|   f'_{e} = 0' , f'|_{C_\rho} = \xi\right) \\
=   \frac{pZ^S_{m-1}(p\ext_{m-1})^{|\cS_1|} ((1-p)Z_{m-1})^{d-|\cS_1|-1} }{  dpZ^S_{m-1} (p\ext_{m-1} + (1-p)Z_{m-1})^{d-1}  + (p\ext_{m-1} + (1-p)Z_{m-1})^{d}}\label{eq:02n}
\end{multline}
Also for any $\cS_1 \subset \cS$,
\begin{multline}
\P'_{n,d,p} \left(f'_{g} =1' \forall g \in \cS_1 , f'_{h} = 0' \forall  h \in \cS \setminus \cS_1) \Big|   f'_{e} = 0' , f'|_{C_\rho} = \xi\right) \\
=\frac{(p\ext_{m-1})^{|\cS_1|} ((1-p)Z_{m-1})^{d-|\cS_1|} }{  dpZ^S_{m-1} (p\ext_{m-1} + (1-p)Z_{m-1})^{d-1}  + (p\ext_{m-1} + (1-p)Z_{m-1})^{d}}\label{eq:01n}
\end{multline}
For any $\cS_1 \subset \cS$,
\begin{multline}
\P'_{n,d,p} \left(f'_{g} =1' \forall g \in \cS_1 , f'_{h} = 0' \forall  h \in \cS \setminus \cS_1) \Big|   f'_{e} = 1' , f'|_{C_\rho} = \xi\right) \\
=\frac{(p\ext_{m-1})^{|\cS_1|} ((1-p)Z_{m-1})^{d-|\cS_1|} }{   (p\ext_{m-1} + (1-p)Z_{m-1})^{d}}\label{eq:1n}
\end{multline}
Finally for all $e_j$ and for all $\cS_1 \subset \cS \setminus \{j\}$.
\begin{multline}
\P'_{n,d,p} \left(f'_{e_j} =2' , f'_{g} =1' \forall g \in \cS_1 , f'_{h} = 0' \forall  h \in \cS \setminus (\cS_1 \cup \{e_j\}) \Big|   f'_{e} = 2' , f'|_{C_\rho} = \xi\right) \\
=   \frac{pZ^S_{m-1}(p\ext_{m-1})^{|\cS_1|} ((1-p)Z_{m-1})^{d-|\cS_1|-1} }{  dpZ^S_{m-1} (p\ext_{m-1} + (1-p)Z_{m-1})^{d-1} }\label{eq:2n}
\end{multline}
In particular, all the expressions in the right hand side above are independent of $\xi$.
\end{lemma}
\begin{proof}
We only calculate the first expression since the rest follow from similar arguments. Write $$\P'_{n,d,p} \left(f'_{e_j} =2' , f'_{g} =1' \forall g \in \cS_1 , f'_{h} = 0' \forall  h \in \cS \setminus (\cS_1 \cup \{e_j\}) \Big|   f'_{e} = 0' , f'|_{C_\rho} = \xi\right) = \frac{N}{D}$$ where 
\begin{align*}
N &= \P'_{n,d,p} \left(f'_{e_j} =2' , f'_{g} =1' \forall g \in \cS_1 , f'_{h} = 0' \forall  h \in \cS \setminus (\cS_1 \cup \{e_j\}) ,  f'_{e} = 0' , f'|_{C_\rho} = \xi\right)\\
D &= \P'_{n,d,p} \left(   f'_{e} = 0' , f'|_{C_\rho} = \xi\right)
\end{align*}
Let us first determine $D$. Note that in any configuration with $f'_e = 0$ and $f'|_{C_{\rho}} = \xi$ must have term $w(\xi)(1-p)$ as a multiplicative factor. Now, for the children of $e$ and their descendants, we have two cases based on whether or not $e_+ \leftrightarrow \partial$.  This gives us the term  
$$dpZ^S_{m-1}\sum_{i = 0}^{d-1}\dbinom{d-1}{i} (p\ext_{m-1})^i((1-p)Z_{m-1})^{d-1-i} + \sum_{i = 0}^d \dbinom{d}{i}(p\ext_{m-1})^i((1-p)Z_{m-1})^{d-i}$$
Applying the binomial theorem, we get  
 $$D = \frac1{Z_n} w(\xi)(1-p)\left[dpZ^S_{m-1}(p\ext_{m-1} + (1-p)Z_{m-1})^{d-1} +  (p\ext_{m-1} + (1-p)Z_{m-1})^{d}\right]$$ 
Now we turn to $N$. Note that we also must include the multiplicative factor $(1-p)w(\xi)$ and the only difference is that the value of $f'$ on the children of $e$ is fixed. The weight of the configuration on the children of $e$ is $$pZ^S_{m-1}(p\ext_{m-1})^{|\cS_1|} ((1-p)Z_{m-1})^{d-|\cS_1|-1} $$ since we must have one surviving path and $\cS_1$ determines the edges with value $1'$.  Therefore, the numerator is $$\frac1{Z_n}(1-p)w(\xi)\left[pZ^S_{m-1}(p\ext_{m-1})^{|\cS_1|} ((1-p)Z_{m-1})^{d-|\cS_1|-1} \right]$$ After simplifications, we see that $\frac{N}{D}$ matches the expression as claimed and furthermore the expression does not depend on $\xi$. The remaining expressions follow from similar calculations. We point out that by definition, none of the children of $1'$ can be a $2'$ and exactly one child of a $2'$ must be a $2'$.
\end{proof}

We now wish to take a limit of the above expressions essentially using our computation in  \cref{cor:survival_prob}. Fix $e$ so that $e_+$ is at graph distance $k$ from $\rho$ for a fixed $k$.
\begin{lemma}\label{lem:Markovian_limit}
Fix $p \in [0,1)$ and $e, \cS, C_\rho$ be as above and let $q = q_{d,p}$ as in \cref{cor:survival_prob}. Suppose $\xi_n$ is a sequence of configurations with $\P'_{n,d,p} (f'|_{C_\rho} = \xi_n) >0$ for all $n$.  Then the following holds for all $e_j$ and for all $\cS_1 \subset \cS \setminus \{e_j\}$.
\begin{multline}
\lim_{n \to \infty }\P'_{n,d,p} \left(f'_{e_j} =2' , f'_{g} =1' \forall g \in \cS_1 , f'_{h} = 0' \forall  h \in \cS \setminus (\cS_1 \cup \{e_j\}) \Big|   f'_{e} = 0' , f'|_{C_\rho} = \xi_n\right) \\
=   \frac{pq (p(1-q))^{|\cS_1|} (1-p)^{d-|\cS_1|-1}}{  dpq (p(1-q) + (1-p))^{d-1}  + (p(1-q) + (1-p))^d}\label{eq:02}
\end{multline}
Also for any $\cS_1 \subset \cS$,
\begin{multline}
\lim_{n \to \infty } \P'_{n,d,p} \left(f'_{g} =1' \forall g \in \cS_1 , f'_{h} = 0' \forall  h \in \cS \setminus \cS_1) \Big|   f'_{e} = 0' , f'|_{C_\rho} = \xi_n\right) \\
=\frac{(p(1-q))^{|\cS_1|} (1-p)^{d-|\cS_1|}}{ dpq (p(1-q) + (1-p))^{d-1}  + (p(1-q) + (1-p))^d}\label{eq:01}
\end{multline}
For any $\cS_1 \subset \cS$,
\begin{equation}
\lim_{n \to \infty } \P'_{n,d,p} \left(f'_{g} =1' \forall g \in \cS_1 , f'_{h} = 0' \forall  h \in \cS \setminus \cS_1) \Big|   f'_{e} = 1' , f'|_{C_\rho} = \xi_n\right) 
=\frac{(p(1-q))^{|\cS_1|} (1-p)^{d-|\cS_1|} }{   (p(1-q) + (1-p))^{d}}\label{eq:1}
\end{equation}
Finally for all $e_j$ and for all $\cS_1 \subseteq \cS \setminus \{j\}$ and $p \in (1/d,1)$,
\begin{multline}
\lim_{n \to \infty } \P'_{n,d,p} \left(f'_{e_j} =2' , f'_{g} =1' \forall g \in \cS_1 , f'_{h} = 0' \forall  h \in \cS \setminus (\cS_1 \cup \{e_j)) \Big|   f'_{e} = 2' , f'|_{C_\rho} = \xi_n\right) \\
=   \frac{pq (p(1-q))^{|\cS_1|} (1-p)^{d-|\cS_1|-1} }{  dpq (p(1-q) + (1-p))^{d-1} }\label{eq:2}
\end{multline}

\end{lemma}
\begin{proof}
For each expression in the right hand side of \cref{lem:Markovian}, we divide the numerator and the denominator by $Z_{m-1}^d$ and use \cref{cor:survival_prob}.
\end{proof}


This motivates us to define the following tree indexed Markov process on $\{0',1',2'\}$ with parameters $\theta, \alpha, \beta$ as follows. Assume the parent of the edges incident to $\rho$ has value $0'$ (to initiate the process).

\begin{itemize}
\item  If an edge $e$ has value $0'$ then toss a coin with success probability $\theta$. 
\begin{itemize}
\item If success occurs we choose uniformly at random one of the children edges of $e$ to be $2'$. Then for each of its siblings, decide to put value $1'$ or $0'$ independently with probability $\alpha$ or $1-\alpha$.
\item  If failure occurs, then for each of its children, decide to put value $1'$ or $0'$ independently with probability $\alpha$ or $1-\alpha$
\end{itemize}

\item If an edge $e$ has value $2'$ choose one of its children uniformly at random and assign value $2'$. For it's siblings, assign value $1'$ or $0'$  with probability $\alpha$, $1-\alpha.$
\item If an edge has value $1'$, assign independently each of it's children value $1'$ or $0'$ with probability $\alpha$ or $1-\alpha$.
\end{itemize}
Note that if $\theta =0$, the above process describes an i.i.d.\ Bernoulli $(\alpha)$ percolation process since we never get a $2'$ in the configuration. Call the probability measure induced by this Markov process by $\P'_{d,p}$.
\begin{thm}\label{thm:limit}
Fix $p \in (1/d,1)$. The probability measures $\P'_{n,d,p}$ converges to the Markov process described above with the following  parameters
$$
\theta = \frac{dpq}{dpq + (1-pq)}  = q \qquad \alpha = \frac{p(1-q)} {   p(1-q) + (1-p)}  = \frac1d.
$$
If $p \le 1/d$, $\P'_{n,d,p}$ converges to an i.i.d.\ Bernoulli edge percolation process with parameter $p$.
\end{thm}
\begin{proof}
Let $f'_n$ denote a sample from $\P'_{n,d,p}$. We first show that $f'_n$ converges in law to $f'$ where $f'$ is a tree indexed Markov process with state space $\{0',1',2'\}$ whose transition probabilities are described by \cref{lem:Markovian_limit} (the chain starts from $0'$ being the parent of $\rho$). In particular, the joint distribution of the children given the parent is $0'$ is described by \eqref{eq:02}, \eqref{eq:01}, that if the parent is $1'$ is described by \eqref{eq:1} and that if the parent is $2'$ is described by \eqref{eq:2}. We see this by induction. The edges attached to the root has the right joint distribution from the limits given by \eqref{eq:02n}, \eqref{eq:02} and \eqref{eq:01n}, \eqref{eq:01} (applied without any conditioning on $\xi_n$). Suppose we sample $f'_n$ step by step, always sampling all the children of an edge in a particular step. Suppose we have sampled $f'_n|_{T'} =\xi_k$ for some subtree $T'$ containing $k$ edges. We pick an edge $e$ in $T'$ whose children are yet to be sampled. Note that even conditioned on all possible extensions of $\xi_k$ to a valid configuration $\xi_n$ on the component of the root in $T_{n,d} \setminus \{e\}$, the joint distribution of all the children of $e$ converge to that described by the Markov chain described in \cref{lem:Markovian,lem:Markovian_limit}. Therefore taking the average over all possible such $\xi_n$ which are extensions of $\xi_k$, the convergence also holds since the conditional distribution does not depend on the extension. Since $T'$ was an arbitrary subtree the limit of $f_n'$ has the above described law.

It remains to show that $f' \sim \P'_{d,p}$. This is a straightforward verification. 
Looking at \eqref{eq:1}, we have the probability of getting $1'$ in $\cS_1$ and $0'$ in the rest is given by $$(\alpha)^{|\cS_1|}(1-\alpha)^{d-|\cS_1|} = \frac{(p(1-q))^{|\cS_1|} (1-p)^{d-|\cS_1|} }{   (p(1-q) + (1-p))^{|\cS_1|}},$$
which matches the description.
For \eqref{eq:02}, the probability of producing the event is given by $$\frac{\theta}{d} \alpha^{|\cS_1|}(1-\alpha)^{d-1-|\cS_1|} $$
which matches the expression on the right hand side of \eqref{eq:02} after plugging in the value. For \eqref{eq:01}, the probability of the event is $$(1-\theta)\alpha^{|\cS_1|}(1-\alpha)^{d-|\cS_1|} $$
which also matches with the right hand side of \eqref{eq:01}. Finally the probability of the event in \eqref{eq:2} in the description is given by
$$
\frac1d \alpha^{|\cS_1|}(1-\alpha)^{d-|\cS_1|}
$$
which also matches with the right hand side of \eqref{eq:2}.
\end{proof}
%

As an immediate consequence, we obtain an improved version of \cref{thm:main} with an explicit description of the limiting process.
\begin{thm}\label{thm:main_improved}
The law $\P_{n,d,p}$ converges in law to pushforward under $\phi^{-1}$ of $\P'_{d,p}$. 
\end{thm}
\begin{proof}
Clearly $\phi^{-1}$ is continuous in the topology of pointwise convergence, from which the result is immediate.
\end{proof}

\begin{proof}[Proof of \cref{thm:description}]
This is immediate from the description of Markov process $\P'_{d,p}$. Indeed if $p \le 1/d$, $\theta =0$, and the process reduces to a Bernoulli $(p)$ percolation. On the other hand if $p \in (1/d,1)$,  since $\alpha>0$, it is clear (for example applying an appropriate Borel--Cantelli) that there are infinitely many infinite components. Each component is one ended since $2'$ produces exactly one $2'$. Furthermore the offsprings of $1'$ which is a sibling of a $2'$ are distributed as a branching process with parameter Bin$(d,1/d)$ (critical branching process). 
\end{proof}
\bibliographystyle{amsplain}
\bibliography{forest}
\end{document}